\documentclass[11pt,reqno]{amsart}
 \usepackage{amssymb,amsfonts,amscd,amsbsy}
\usepackage[mathscr]{eucal}
\usepackage{url}

\newtheorem{theorem}{Theorem}[section]
\newtheorem{lemma}[theorem]{Lemma}
\newtheorem{proposition}[theorem]{Proposition}
\newtheorem{corollary}[theorem]{Corollary}
\theoremstyle{definition}

\numberwithin{equation}{section}

\makeatletter
\def\imod#1{\allowbreak\mkern5mu({\operator@font mod}\,\,#1)}
\makeatother
\allowdisplaybreaks

\begin{document}

\title[Mock theta double sums]
{Mock theta double sums} 
 
\author{Jeremy Lovejoy}

\author{Robert Osburn}

\address{CNRS, LIAFA, Universit{\'e} Denis Diderot - Paris 7, Case 7014, 75205 Paris Cedex 13, FRANCE}

\address{School of Mathematical Sciences, University College Dublin, Belfield, Dublin 4, Ireland}

\address{IH{\'E}S, Le Bois-Marie, 35, route de Chartres, F-91440 Bures-sur-Yvette, FRANCE}

\email{lovejoy@liafa.jussieu.fr}

\email{robert.osburn@ucd.ie, osburn@ihes.fr}

\subjclass[2010]{Primary: 33D15; Secondary: 11F37}
\keywords{Bailey pairs, mock theta functions, identities}

\date{\today}

\begin{abstract}
We prove a general result on Bailey pairs and show that two Bailey pairs of Bringmann and Kane are special cases.  We also show how to use a change of base formula to pass from the pairs of Bringmann and Kane to pairs used by Andrews in his study of Ramanujan's seventh order mock theta functions.  We derive several more Bailey pairs of a similar type and use these to construct a number of new $q$-hypergeometric double sums which are mock theta functions.  Finally, we prove identities between some of these mock theta double sums and classical mock theta functions. 
\end{abstract}
 
\maketitle

\section{Introduction}
\subsection{Bailey pairs}
A \emph{Bailey pair} relative to $a$ is a pair of sequences $(\alpha_n,\beta_n)_{n \geq 0}$ satisfying
\begin{equation} \label{pairdef}
\beta_n = \sum_{k=0}^n \frac{\alpha_k}{(q)_{n-k}(aq)_{n+k}}, 
\end{equation} 
or equivalently 
\begin{equation} \label{pairdefbis}
\alpha_n = (1-aq^{2n}) \sum_{j=0}^n \frac{(aq)_{n+j-1}(-1)^{n-j}q^{\binom{n-j}{2}}}{(q)_{n-j}}\beta_j.
\end{equation}
Here we have used the standard $q$-hypergeometric notation, 
\begin{equation*}
(a)_n = (a;q)_n = \prod_{k=1}^{n} (1-aq^{k-1}),
\end{equation*}
valid for $n \in \mathbb{N} \cup \{\infty\}$.   The \emph{Bailey lemma} says that if $(\alpha_n,\beta_n)$ is a Bailey pair relative to $a$, then so is $(\alpha_n',\beta_n')$, where 

\begin{equation} \label{alphaprimedef}
\alpha'_n = \frac{(\rho_1)_n(\rho_2)_n(aq/\rho_1 \rho_2)^n}{(aq/\rho_1)_n(aq/\rho_2)_n}\alpha_n
\end{equation} 

\noindent and

\begin{equation} \label{betaprimedef}
\beta'_n = \sum_{k=0}^n\frac{(\rho_1)_k(\rho_2)_k(aq/\rho_1 \rho_2)_{n-k} (aq/\rho_1 \rho_2)^k}{(aq/\rho_1)_n(aq/\rho_2)_n(q)_{n-k}} \beta_k.
\end{equation}

A useful limiting form of the Bailey lemma is found by putting \eqref{alphaprimedef} and \eqref{betaprimedef} into \eqref{pairdef} and letting $n \to \infty$, giving

\begin{equation} \label{limitBailey}
\sum_{n \geq 0} (\rho_1)_n(\rho_2)_n (aq/\rho_1 \rho_2)^n \beta_n = \frac{(aq/\rho_1)_{\infty}(aq/\rho_2)_{\infty}}{(aq)_{\infty}(aq/\rho_1 \rho_2)_{\infty}} \sum_{n \geq 0} \frac{(\rho_1)_n(\rho_2)_n(aq/\rho_1 \rho_2)^n }{(aq/\rho_1)_n(aq/\rho_2)_n}\alpha_n.
\end{equation}
For more on Bailey pairs and the Bailey lemma, see \cite{An1,An2,war}.

This paper has its origins in the following two Bailey pairs discovered by Bringmann and Kane \cite{Br-Ka1}.  First, $(a_n,b_n)$ is a Bailey pair relative to $1$, where

\begin{equation} \label{a2n}
a_{2n} = (1-q^{4n})q^{2n^2-2n}\sum_{j=-n}^{n-1}q^{-2j^2-2j},
\end{equation}

\begin{equation} \label{a2n+1}
a_{2n+1} = -(1-q^{4n+2})q^{2n^2}\sum_{j=-n}^{n} q^{-2j^2},
\end{equation}

\noindent and

\begin{equation} \label{bn}
b_n = \frac{(-1)^n(q;q^2)_{n-1}}{(q)_{2n-1}} \chi(\text{$n \neq 0$}) \footnote{As usual, $\chi(X) = 1$ if $X$ is true and $0$ if $X$ is false.},
\end{equation}

\noindent and second, $(\alpha_n,\beta_n)$ is a Bailey pair relative to $q$, where

\begin{equation} \label{alpha2n}
\alpha_{2n} = \frac{1}{1-q}\left(q^{2n^2+2n}\sum_{j=-n}^{n-1}q^{-2j^2-2j} + q^{2n^2}\sum_{j=-n}^{n} q^{-2j^2}\right),
\end{equation}

\begin{equation} \label{alpha2n+1}
\alpha_{2n+1} = -\frac{1}{1-q}\left(q^{2n^2+4n+2}\sum_{j=-n}^{n} q^{-2j^2} + q^{2n^2+2n}\sum_{j=-n-1}^n q^{-2j^2-2j}\right),
\end{equation}

\noindent and 

\begin{equation} \label{betan}
\beta_n = \frac{(-1)^n(q;q^2)_n}{(q)_{2n+1}}.
\end{equation}

These are highly reminiscent of three Bailey pairs discovered by Andrews \cite{An3} in his study of Ramanujan's seventh order mock theta functions.  Namely, he showed that $(\mathcal{A}_n(0),\mathcal{B}_n(0))$ and $(\mathcal{A}_n(1),\mathcal{B}_n(1))$ form Bailey pairs relative to $1$, where 
\begin{equation} \label{mathcalA2n(0)}
\mathcal{A}_{2n}(0) = q^{3n^2+n}\sum_{j=-n}^{n}q^{-j^2} - q^{3n^2-n}\sum_{j=-n+1}^{n-1} q^{-j^2},
\end{equation}

\begin{equation} \label{mathcalA2n+1(0)}
\mathcal{A}_{2n+1}(0) = -q^{3n^2+4n+1}\sum_{j=-n-1}^{n} q^{-j^2-j} + q^{3n^2+2n}\sum_{j=-n}^{n-1} q^{-j^2-j},
\end{equation}

\begin{equation} \label{mathcalBn(0)}
\mathcal{B}_n(0) = \frac{1}{(q^{n+1})_{n}},
\end{equation}

\begin{equation} \label{mathcalA2n(1)}
\mathcal{A}_{2n}(1) = - (1-q^{4n})q^{3n^2-2n}\sum_{j=-n}^{n-1}q^{-j^2-j},
\end{equation}
\begin{equation} \label{mathcalA2n+1(1)}
\mathcal{A}_{2n+1}(1) = (1-q^{4n+2})q^{3n^2+n}\sum_{j=-n}^{n}q^{-j^2},
\end{equation}
and
\begin{equation} \label{mathcalBn(1)}
\mathcal{B}_n(1) = \frac{1}{(q^n)_n}\chi(n \neq 0),
\end{equation}
while $(\mathcal{A}_n(2),\mathcal{B}_n(2))$ is a Bailey pair relative to $q$, where

\begin{equation} \label{mathcalA2n(2)}
\mathcal{A}_{2n}(2) = \frac{1}{1-q}\left(q^{3n^2+2n}\sum_{j=-n}^{n-1}q^{-j^2-j} + q^{3n^2+n}\sum_{j=-n}^{n} q^{-j^2}\right),
\end{equation}

\begin{equation} \label{mathcalA2n+1(2)}
\mathcal{A}_{2n+1}(2) = -\frac{1}{1-q}\left(q^{3n^2+5n+2}\sum_{j=-n}^{n} q^{-j^2} + q^{3n^2+4n+1}\sum_{j=-n-1}^n q^{-j^2-j}\right),
\end{equation}

\noindent and 

\begin{equation} \label{mathcalBn(2)}
\mathcal{B}_n(2) = \frac{1}{(q^{n+1})_{n+1}}.
\end{equation}

Our first goal in this paper is to prove the following results, which will lead to more Bailey pairs like those of Bringmann-Kane and Andrews.  Note that Theorem \ref{main3} is simply an application of Theorem \ref{main1} followed by an application of Theorem \ref{main2}.  
\begin{theorem} \label{main1}
If $(\alpha_n,\beta_n)$ is a Bailey pair relative to $1$ with $\alpha_0 = \beta_0=1$, then $(\alpha'_n,\beta'_n)$ is also a Bailey pair relative to $1$, where $\alpha'_0=\beta'_0 = 0$,
\begin{equation} \label{aeven}
\alpha'_{2n} = -(1-q^{4n})q^{2n^2-2n}\sum_{j=0}^{n-1}q^{-2j^2-2j}\alpha_{2j+1},
\end{equation}
\begin{equation} \label{aodd}
\alpha'_{2n+1} = -(1-q^{4n+2})q^{2n^2}\sum_{j=0}^{n}q^{-2j^2}\alpha_{2j},
\end{equation}
and for $n \geq 1$,
\begin{equation} \label{b1}
\beta'_n = -\frac{\beta_{n-1}}{1-q^{2n-1}}.
\end{equation}
\end{theorem}

\begin{theorem} \label{main2}
Suppose that $(\alpha_n,\beta_n)$ is a Bailey pair relative to $1$ with $\alpha_0=\beta_0=0$.   Then $(\alpha'_n,\beta'_n)$ is a Bailey pair relative to $q$, where
\begin{equation} \label{a2}
\alpha'_n = \frac{1}{1-q}\left(-\frac{\alpha_{n+1}}{1-q^{2n+2}} + \frac{q^{2n}\alpha_n}{1-q^{2n}}\right)
\end{equation}
and
\begin{equation} \label{b2}
\beta'_n = -\beta_{n+1}.
\end{equation}
\end{theorem}

\begin{theorem} \label{main3}
If $(\alpha_n,\beta_n)$ is a Bailey pair relative to $1$ with $\alpha_0 = \beta_0=1$, then $(\alpha''_n,\beta''_n)$ is a Bailey pair relative to $q$, where
\begin{equation} 
\alpha''_{2n} = \frac{1}{1-q}\left(q^{2n^2}\sum_{j=0}^{n}q^{-2j^2}\alpha_{2j} - q^{2n^2+2n}\sum_{j=0}^{n-1}q^{-2j^2-2j}\alpha_{2j+1}\right),
\end{equation}
\begin{equation} 
\alpha''_{2n+1} = \frac{1}{1-q}\left(q^{2n^2+2n}\sum_{j=0}^{n}q^{-2j^2-2j}\alpha_{2j+1} - q^{2n^2+4n+2}\sum_{j=0}^{n}q^{-2j^2}\alpha_{2j}\right),
\end{equation}
and
\begin{equation}
\beta''_n = \frac{\beta_{n}}{1-q^{2n+1}}.
\end{equation}
\end{theorem}

An application of Theorem \ref{main1} or \ref{main3} to a ``typical" Bailey pair (from Slater's list \cite{Sl1}, for example) will give a positive definite quadratic form in the power of $q$ occurring in $\alpha_n$.   However, there are a few cases where we obtain an indefinite quadratic form.   For example, using Theorems \ref{main1} and \ref{main3} and the following Bailey pair relative to $1$ from Slater's list \cite[p. 468]{Sl1}, 

\begin{equation*}
\alpha_n=
\begin{cases}
1, &\text{if $n=0$}, \\
2(-1)^n, &\text{otherwise}
\end{cases}
\end{equation*}

\noindent and

\begin{equation*}
\beta_n = \frac{(-1)^n}{(q^2;q^2)_n},
\end{equation*}
we recover the Bailey pairs of Bringmann and Kane in \eqref{a2n}--\eqref{betan}.  Some other examples are recorded in Corollaries \ref{paircor1}--\ref{paircor3}.  Andrews' Bailey pairs in \eqref{mathcalA2n(1)}--\eqref{mathcalBn(2)} do not seem to be simple applications of Theorems \ref{main1}--\ref{main3}, but they can be deduced from \eqref{a2n}--\eqref{betan} using a change of base formula.   We discuss this in Section \ref{Andrews'pairs}.   For another treatment of these pairs, see \cite{An4}.
 
\subsection{Mock theta functions}

An important difference between the pairs of Andrews and those of Bringmann-Kane is that the former yield mock theta functions when substituted into \eqref{limitBailey}, while the latter do not.   However, as we showed in \cite{Lo-Os1}, the Bailey pairs of Bringmann and Kane \emph{do} give rise to mock theta functions after an appropriate application of the Bailey lemma.  These mock theta functions are $q$-hypergeometric double sums.

To recall them, we need some special functions.   We use 
the classical theta series
\begin{equation*} \label{j}
j(x,q):= \sum_{n \in \mathbb{Z}} (-x)^{n}q^{\binom{n}{2}} = (x)_{\infty} (q/x)_{\infty} (q)_{\infty},
\end{equation*} 
and for brevity, we write 
\noindent $J_{m}:= J_{m, 3m}$ with $J_{a,m} := j(q^{a}, q^{m})$, and $\overline{J}_{a,m}:=j(-q^{a}, q^{m})$. 
We also use the Hecke-type series
\begin{equation} \label{fdef}
f_{a,b,c}(x,y,q) : = \left(\sum_{r,s \geq 0} - \sum_{r,s < 0}\right) (-1)^{r+s} x^r y^s q^{a \binom{r}{2} + brs + c \binom{s}{2}},
\end{equation} 
which is an indefinite theta series when $ac < b^2$.   Here, we assume $a$, $c>0$. Finally, we employ the Appell-Lerch series
\begin{equation} \label{Appell-Lerch} 
m(x,q,z) := \frac{1}{j(z,q)} \sum_{r \in \mathbb{Z}} \frac{(-1)^r q^{\binom{r}{2}} z^r}{1-q^{r-1} xz},
\end{equation}

\noindent where $x$, $z \in \mathbb{C}^{*}:=\mathbb{C} \setminus \{ 0 \}$ with neither $z$ nor $xz$ an integral power of $q$.  The Appell-Lerch series is a ``mock Jacobi form" which specializes to a mock theta function when $x$ and $z$ are of the form $\zeta q^{\frac{m}{n}}$ with $\zeta$ a root of unity \cite{Za1,Zw1}.  Recall that a mock theta function is the holomorphic part of a weight $1/2$ harmonic weak Maass form $f(\tau)$ (as usual, $q := e^{2 \pi i \tau}$ where $\tau=x + iy \in \mathbb{H}$) whose image of under the operator $\xi_{\frac{1}{2}} := 2iy^{\frac{1}{2}} \frac{\overline{\partial}}{\partial \overline{\tau}}$ is a unary theta function \cite{On1,Za1}.  

The main result in \cite{Lo-Os1} contains identities equivalent to the following.
\begin{theorem}{\cite[Theorem 1.3]{Lo-Os1}}
\begin{align}
\mathcal{W}_1(q) &:= \sum_{n \geq 1} \sum_{n \geq j \geq 1} \frac{(-1)_j (q; q^2)_{j-1} (-1)^{j} q^{n^2 + \binom{j+1}{2}}}{(-q)_n (q)_{n-j} (q)_{2j-1}} \nonumber \\
&= \frac{-2q^2}{(q)_{\infty}}f_{3,5,3}(q^5,q^5,q) \nonumber \\
&= 4m(-q^{17},q^{48},q^{24}) - 4q^{-5}m(-q,q^{48},q^{24}) - 2q^2\frac{J_8J_{12}J_{96}J_{7,16}\overline{J}_{4,24}J_{6,48}J_{30,96}}{J_{24}J_{48}J_{3,8}J_{2,12}J_{14,96}J_{46,96}}, \label{mockW1} \\
\mathcal{W}_2(q) &:= {\sum\limits_{n \geq 1}}^{*} \sum_{n \geq j \geq 1}\frac{(q; q^2)_n (-1)_{j} (q; q^2)_{j-1} (-1)^{n+j} q^{\binom{j+1}{2}}}{(-q)_{n} (q)_{n-j} (q)_{2j-1}} \nonumber \\
&= \frac{q(q;q^2)_{\infty}}{(q^2;q^2)_{\infty}}f_{1,3,1}(-q^2,-q^2,q) \nonumber \\
&= 4m(-q,q^8,q^4) + q\frac{J_{1,8}^2J_{3,8}^3J_{2,16}}{J_8^4J_{16}}, \label{mockW2} \\ 
\mathcal{W}_3(q) &:=  \sum_{n \geq 1} \sum_{n \geq j \geq 1} \frac{(q; q^2)_{n} (-1;q^2)_j (q^2; q^4)_{j-1} (-1)^{n+j} q^{n^2 + j^2 + j}}{(-q^2; q^2)_{n} (q^2; q^2)_{n-j} (q^2; q^2)_{2j-1}}\nonumber \\
&= \frac{2q^3(q;q^2)_{\infty}}{(q^2;q^2)_{\infty}}f_{1,2,1}(-q^7,-q^7,q^4) \nonumber \\
&= 4m(-q,q^{12},q^4) + 2q^3\frac{\overline{J}_{1,12}^2}{\overline{J}_{1,4}}, \label{mockW3} \\
\mathcal{W}_4(q) &:= \sum_{n \geq 0} \sum_{n \geq j \geq 0} \frac{(-q)_{j} (q; q^2)_{j} (-1)^j q^{n^2 + n + \binom{j+1}{2}}}{(-q)_n (q)_{n-j} (q)_{2j+1}} \nonumber \\
&= \frac{1}{(q)_{\infty}}f_{3,5,3}(q^3,q^3,q) \nonumber \\
&= -2q^{-4}m(-q^5,q^{48},q^{24}) - 2q^{-2}m(-q^{11},q^{48},q^{24}) + \frac{J_8J_{12}J_{96}J_{3,16}\overline{J}_{4,24}J_{6,48}J_{18,96}J_{30,96}}{J_{24}J_{48}J_{1,8}J_{2,12}J_{6,96}J_{26,96}J_{38,96}}. \label{mockW4} 
\end{align}
\end{theorem}

In particular, $\mathcal{W}_1(q)$--$\mathcal{W}_4(q)$ are mock theta functions.  We remark that the series defining $\mathcal{W}_{2}(q)$ does not converge.  However, similar to the classical sixth order mock theta function $\mu(q)$ \cite{AH}, the sequence of even partial sums and the sequence of odd partial sums both converge.  We define $\mathcal{W}_{2}(q)$ as the average of these two values.   This averaging is denoted here and throughout by the notation ${\sum\limits_{}}^{*}$. 

The second goal of this paper is to use Bailey pairs arising from Theorems \ref{main1} and \ref{main3} to obtain many more mock theta functions like $\mathcal{W}_1(q)$--$\mathcal{W}_4(q)$.   Just as with the pairs of Bringmann and Kane, this first requires one application of the Bailey lemma, and so the mock theta functions we obtain are $q$-hypergeometric double sums.  We record these mock theta functions in three separate results, corresponding to three sets of Bailey pairs.   We first express the double sums in terms of the indefinite theta series \eqref{fdef} and then in terms of the Appell-Lerch series \eqref{Appell-Lerch}.

\begin{theorem} \label{mockthm1}
The following are mock theta functions. 
\begin{align}
\mathcal{M}_{1}(q) &:=  \sum_{n \geq 1} \sum_{n \geq j \geq 1} \frac{(-1)_j(-1)^jq^{n^2+\binom{j}{2}}}{(-q)_n(q)_{n-j}(q^2;q^2)_{j-1}(1-q^{2j-1})}  \nonumber \\
&= -\frac{2q}{(q)_{\infty}}f_{3,5,3}(q^4,q^6,q) \nonumber \\
&= 4q^{-3} m(-q^7, q^{48}, q^{24}) + 4m(-q^{25}, q^{48}, q^{-24}) - 2 - 2\frac{J_8J_{12}J_{96}J_{1,16}\overline{J}_{4,24}J_{6,48}J_{18,96}}{J_{24}J_{48}J_{3,8}J_{2,12}J_{2,96}J_{34,96}}, \label{mock1-1} \\
\mathcal{M}_{2}(q) &:= {\sum\limits_{n \geq 1}}^{*} \sum_{n \geq j \geq 1} \frac{(q;q^2)_n(-1)_j(-1)^{n+j}q^{\binom{j}{2}}}{(-q)_n(q)_{n-j}(q^2;q^2)_{j-1}(1-q^{2j-1})} \nonumber \\
&= \frac{(q;q^2)_{\infty}}{(q^2;q^2)_{\infty}}f_{1,3,1}(-q,-q^3,q) \nonumber \\
&= 4m(-q^5,q^8,q^4) -2 - \frac{J_{1,8}^3J_{3,8}^2J_{6,16}}{J_8^4J_{16}},  \label{mock1-2} \\
\mathcal{M}_{3}(q) &:=  \sum_{n \geq 1} \sum_{n \geq j \geq 1} \frac{(q;q^2)_n(-1;q^2)_j(-1)^{n+j}q^{n^2+j^2-j}}{(-q^2;q^2)_n(q^2;q^2)_{n-j}(q^4;q^4)_{j-1}(1-q^{4j-2})} \nonumber \\
&=  \frac{2q(q;q^2)_{\infty}}{(q^2;q^2)_{\infty}}f_{1,2,1}(-q^5,-q^9,q^4) \nonumber \\
&= 4m(-q^7, q^{12},q^4) - 2 + 2\frac{J_{12}^3\overline{J}_{5,12}}{J_{4,12}\overline{J}_{1,12}\overline{J}_{3,12}},  \label{mock1-3} \\
\mathcal{M}_{4}(q) &:= \sum_{n \geq 0} \sum_{n \geq j \geq 0} \frac{(-1)^jq^{n^2+n+\binom{j}{2}}}{(-q)_n(q)_{n-j}(q)_{j}(1-q^{2j+1})} \nonumber \\
&= \frac{1}{(q)_{\infty}}f_{3,5,3}(q^2,q^4,q) \nonumber \\
&= 2m(-q^{29}, q^{48},q^{24}) - 1 - 2q^{-1} m(-q^{13}, q^{48}, q^{-24}) +q\frac{J_8J_{12}J_{96}^3J_{5,16}\overline{J}_{4,24}J_{6,48}J_{18,48}}{J_{24}J_{48}^2J_{1,8}J_{2,12}J_{10,96}J_{22,96}J_{42,96}}, \label{mock1-4} \\
\mathcal{M}_{5}(q) &:= \sum_{n \geq 0} \sum_{n \geq j \geq 0}  \frac{(-1)^{j}q^{\binom{n+1}{2}+\binom{j}{2}}}{(q)_{n-j}(q)_{j}(1-q^{2j+1})}\nonumber  \\
&= \frac{(-q)_{\infty}}{(q)_{\infty}} f_{1,2,1}(q,q^3,q^2) \nonumber \\
&= 2m(q^5, q^6,q^2) - 1 -q\frac{J_6^3}{J_{2,6}J_{3,6}}. \label{mock1-5}
\end{align}
\end{theorem}

\begin{theorem} \label{mockthm2}
The following are mock theta functions.  
\begin{align}
\mathcal{M}_{6}(q) &:= \sum_{n \geq 1} \sum_{n \geq j \geq 1} \frac{(-1)^jq^{n^2+\binom{j+1}{2}}}{(q)_{n-j}(q)_{j-1}(1-q^{2j-1})} \nonumber \\
&= -\frac{q^2}{(q)_{\infty}}f_{3,7,3}(q^5,q^6,q) \nonumber \\
&=  m(-q^{49}, q^{120}, q^{-3}) - q^{-3} m(-q^{89}, q^{120}, q^{-3}) + q^{-3} + q^{-14} m(-q^{119}, q^{120}, q^3) - q^{-14} \nonumber \\
&- q^{-1} m(-q^{79}, q^{120}, q^3) + q^{-1} - q^{-11}\frac{J_{12,48}J_{16,40}J_{2,20}J_{3,40}\overline{J}_{17,40}J_{40}}{J_1J_{3,120}\overline{J}_{6,40}J_{20}J_{80}} \nonumber \\
& + q^{-4}\frac{J_{24,48}J_{1,40}J_{4,40}\overline{J}_{1,40}J_{8,20}\overline{J}_{4,40}J_{18,40}J_{80}}{J_1J_{3,120}\overline{J}_{6,40}\overline{J}_{2,40}J_{20}^2J_{40}} \nonumber \\
&+ q^{-12}\frac{J_{24,48}J_{1,40}J_{4,40}\overline{J}_{1,40}J_{8,20}\overline{J}_{16,40}J_{42,80}^2}{J_1J_{3,120}\overline{J}_{6,40}\overline{J}_{2,40}J_{20}^2J_{80}}, \label{mock2-1} \\
\mathcal{M}_{7}(q) 
&:= \sum_{n \geq 1} \sum_{n \geq j \geq 1} \frac{(-1)_n(-1)^jq^{\binom{n+1}{2}+ \binom{j+1}{2}}}{(q)_{n-j}(q)_{j-1}(1-q^{2j-1})} \nonumber \\
&= -\frac{2q^2(-q)_{\infty}}{(q)_{\infty}}f_{1,3,1}(q^4,q^5,q^2) \nonumber \\
&= 2q^{-1} m(-q, q^{16}, q^{-1}) - 2q^{-1}\frac{J_{4,8}J_{16,32}J_{1,16}J_{14,32}}{J_{1,2}\overline{J}_{2,16}\overline{J}_{0,16}}, \label{mock2-2} \\
\mathcal{M}_{8}(q) 
&:= 2{\sum\limits_{n \geq 1}}^{*} \sum_{n \geq j \geq 1} \frac{(q;q^2)_n(-1)^{n+j}q^{\binom{j+1}{2}}}{(q)_{n-j}(q)_{j-1}(1-q^{2j-1})} \nonumber \\
&= \frac{q(q;q^2)_{\infty}}{(q^2;q^2)_{\infty}}f_{1,5,1}(-q^2,-q^3,q) \nonumber \\
&= 2m(-q^7,q^{24},q^6) + 2q^{-2}m(-q,q^{24},q^{-6}) + q \frac{J_1J_{3,8}J_{2,16}}{J_2J_{16}}, \\
\mathcal{M}_{9}(q) 
&:= \sum_{n \geq 1} \sum_{n \geq j \geq 1}\frac{(q;q^2)_n(-1)^{n+j}q^{n^2+j^2+j}}{(q^2;q^2)_{n-j}(q^2;q^2)_{j-1}(1-q^{4j-2})} \nonumber \\
&= \frac{q^3(q;q^2)_{\infty}}{(q^2;q^2)_{\infty}}f_{1,3,1}(-q^7,-q^9,q^4)  \nonumber \\
&= m(-q^8, q^{32}, q^{-2}) - q^{-1}\frac{J_{64}^2J_{28,64}}{J_{32}J_{4,64}} + q^{-1}\frac{J_{8,16}J_{32,64}J_{4,32}J_{24,64}}{\overline{J}_{1,4}\overline{J}_{6,32}\overline{J}_{2,32}},  \label{mock2-4} \\
\mathcal{M}_{10}(q) 
&:= \sum_{n \geq 1} \sum_{n \geq j \geq 1}\frac{(-1)_{2n}(-1)^{j}q^{n+j^2+j}}{(q^2;q^2)_{n-j}(q^2;q^2)_{j-1}(1-q^{4j-2})} \nonumber \\
&=  -\frac{2q^3(-q)_{\infty}}{(q)_{\infty}}f_{1,5,1}(q^5,q^7,q^2) \nonumber \\ 
&=  -2q^{-1} m(-q^{10}, q^{48}, q^{-2}) -2q^{-4} m(-q^2, q^{48}, q^{-2}) - 4q^{-3}\frac{J_{8,32}J_{20,48}\overline{J}_{22,48}J_{2,24}J_{6,48}J_{96}}{J_{1,2}\overline{J}_{8,48}\overline{J}_{0,48}J_{24}J_{2,48}} \nonumber \\ 
&+ 2q^{6}\frac{J_{16,32}J_{4,48}\overline{J}_{2,48}J_{10,24}\overline{J}_{4,48}J_{20,48}J_{96}}{J_{1,2}\overline{J}_{8,48}\overline{J}_{0,48}J_{24}^2J_{48}} + 2q^{-4}\frac{J_{16,32}J_{4,48}\overline{J}_{2,48}J_{10,24}\overline{J}_{20,48}J_{44,96}^2}{J_{1,2}\overline{J}_{8,48}\overline{J}_{0,48}J_{24}^2J_{96}}, \label{mock2-5} \\
\mathcal{M}_{11}(q) 
&:= \sum_{n \geq 0} \sum_{n \geq j \geq 0} \frac{(-1)^jq^{n^2+n+\binom{j+1}{2}}}{(q)_{n-j}(q)_{j}(1-q^{2j+1})} \nonumber \\
&= \frac{1}{(q)_{\infty}}f_{3,7,3}(q^3,q^4,q) \nonumber \\
&= q^{-8} m(-q^{17}, q^{120}, q^{-3}) + q^{-6} m(-q^{23}, q^{120}, q^3) - q^{-1} m(-q^{47}, q^{120}, q^3) \nonumber \\
&+ q^{-12} m(-q^7, q^{120}, q^3) + q^{-9}\frac{J_{12,48}J_{1,40}J_{8,40}\overline{J}_{19,40}J_{6,20}\overline{J}_{12,40}J_{18,40}J_{40}^2}{J_1J_{3,120}\overline{J}_{14,40}\overline{J}_{10,40}J_{20}^3J_{80}} \nonumber \\
&+q^{-4}\frac{J_{12,48}J_{1,40}J_{8,40}\overline{J}_{19,40}J_{6,20}\overline{J}_{8,40}J_{19,40}^2\overline{J}_{1,40}^2}{J_1J_{3,120}\overline{J}_{14,40}\overline{J}_{10,40}J_{20}^3J_{40}J_{80}} \nonumber \\
& - q^{-4}\frac{J_{24,48}J_{1,40}J_{12,40}\overline{J}_{1,40}J_{4,20}\overline{J}_{12,40}J_{18,40}J_{80}}{J_1J_{3,120}\overline{J}_{14,40}\overline{J}_{10,40}J_{20}^2J_{40}} - q^{-8}\frac{J_{24,48}J_{1,40}J_{12,40}\overline{J}_{1,40}J_{4,20}\overline{J}_{8,40}J_{38,80}^2}{J_1J_{3,120}\overline{J}_{14,40}\overline{J}_{10,40}J_{20}^2J_{80}}  , \label{mock2-6} \\
\mathcal{M}_{12}(q) 
&:= \sum_{n \geq 0} \sum_{n \geq j \geq 0} \frac{(-q)_n(-1)^jq^{\binom{n+1}{2} + \binom{j+1}{2}}}{(q)_{n-j}(q)_{j}(1-q^{2j+1})} \nonumber \\
&= \frac{(-q)_{\infty}}{(q)_{\infty}} f_{1,3,1}(q^2,q^3,q^2) \nonumber  \\
&=  -q^{-1} m(-q^3, q^{16}, q) + q\frac{J_{4,8}J_{16,32}J_{5,16}J_{6,32}}{J_{1,2}\overline{J}_{6,16}\overline{J}_{4,16}}. \label{mock2-7} 
\end{align}
\end{theorem}

\begin{theorem} \label{mockthm3}
The following are mock theta functions.
\begin{align}
\mathcal{M}_{13}(q) &:= \sum_{n \geq 1} \sum_{n \geq j \geq 1} \frac{(-1)^jq^{n^2+\binom{j}{2}}}{(q)_{n-j}(q)_{j-1}(1-q^{2j-1})} \nonumber \\
&= -\frac{q}{(q)_{\infty}}f_{3,7,3}(q^4,q^7,q) \nonumber \\
&= m(-q^{59}, q^{120}, q^{-9}) - q^{-7} m(-q^{19}, q^{120}, q^{-9}) - q^{-4} m(-q^{29}, q^{120}, q^9) \nonumber \\
&- q^{-10} m(-q^{11}, q^{120}, q^{-9}) - q^{-8}\frac{J_{12,48}J_{3,40}J_{16,40}\overline{J}_{17,40}J_{2,20}\overline{J}_{4,40}J_{14,40}J_{40}^2}{J_1J_{9,120}\overline{J}_{10,40}\overline{J}_{2,40}J_{20}^3J_{80}} \nonumber \\
&-q^{-9}\frac{J_{12,48}J_{3,40}J_{16,40}\overline{J}_{17,40}J_{2,20}\overline{J}_{16,40}J_{17,40}^2\overline{J}_{3,40}^2}{J_1J_{9,120}\overline{J}_{10,40}\overline{J}_{2,40}J_{20}^3J_{40}J_{80}} \nonumber \\
&+ q^{-2}\frac{J_{24,48}J_{3,40}J_{4,40}\overline{J}_{3,40}J_{8,20}\overline{J}_{4,40}J_{14,40}J_{80}}{J_1J_{9,120}\overline{J}_{10,40}\overline{J}_{2,40}J_{20}^2J_{40}} \nonumber \\
&+ q^{-10}\frac{J_{24,48}J_{3,40}J_{4,40}\overline{J}_{3,40}J_{8,20}\overline{J}_{16,40}J_{34,80}^2}{J_1J_{9,120}\overline{J}_{10,40}\overline{J}_{2,40}J_{20}^2J_{80}},  \label{mock3-1} \\
\mathcal{M}_{14}(q) 
&:= \sum_{n \geq 1} \sum_{n \geq j \geq 1} \frac{(-1)_n(-1)^jq^{\binom{n+1}{2}+\binom{j}{2}}}{(q)_{n-j}(q)_{j-1}(1-q^{2j-1})} \nonumber \\
&= -\frac{2q(-q)_{\infty}}{(q)_{\infty}}f_{1,3,1}(q^3,q^6,q^2) \nonumber \\
&= 2m(-q^7, q^{16}, q^{-3}) - 2\frac{J_{4,8}J_{16,32}J_{1,16}J_{14,32}}{J_{1,2}\overline{J}_{2,16}\overline{J}_{4,16}},  \label{mock3-2} \\
\mathcal{M}_{15}(q) 
&:=2{\sum\limits_{n \geq 1}}^{*} \sum_{n \geq j \geq 1} \frac{(q;q^2)_n(-1)^{n+j}q^{\binom{j}{2}}}{(q)_{n-j}(q)_{j-1}(1-q^{2j-1})} \nonumber \\
&= \frac{(q;q^2)_{\infty}}{(q^2;q^2)_{\infty}}f_{1,5,1}(-q,-q^4,q) \nonumber \\
&= 2m(-q^{13}, q^{24}, q^{2}) - 2q^{-1} m(-q^5, q^{24}, q^{2}) + \frac{J_1J_{1,8}J_{6,16}}{J_2J_{16}}, \\
\mathcal{M}_{16}(q) 
&:= \sum_{n \geq 1} \sum_{n \geq j \geq 1} \frac{(q;q^2)_n(-1)^{n+j}q^{n^2+j^2-j}}{(q^2;q^2)_{n-j}(q^2;q^2)_{j-1}(1-q^{4j-2})} \nonumber \\
&=  \frac{q(q;q^2)_{\infty}}{(q^2;q^2)_{\infty}}f_{1,3,1}(-q^5,-q^{11},q^4) \nonumber \\
&= - q^{-1} m(-q^8, q^{32}, q^{-6}) + \frac{1}{2} + \frac{J_{32}^3 J_{10,32}  \overline{J}_{6,32}}{J_{6,32} J_{16,32} \overline{J}_{0,32} \overline{J}_{10,32}} + q^{-1}\frac{J_{8,16}J_{32,64}J_{4,32}J_{24,64}}{\overline{J}_{1,4}\overline{J}_{2,32}\overline{J}_{10,32}}, \label{mock3-4} \\
\mathcal{M}_{17}(q) 
&:= \sum_{n \geq 1} \sum_{n \geq j \geq 1}\frac{(-1)_{2n}(-1)^{j}q^{n+j^2-j}}{(q^2;q^2)_{n-j}(q^2;q^2)_{j-1}(1-q^{4j-2})} \nonumber \\ 
&= -\frac{2q(-q)_{\infty}}{(q)_{\infty}}f_{1,5,1}(q^3,q^9,q^2) \nonumber \\
& =  2m(-q^{22}, q^{48}, q^{-6}) + 2q^{-1} m(-q^{14}, q^{48}, q^{-6}) - 2q^{3}\frac{J_{8,32}J_{20,48}\overline{J}_{18,48}J_{2,24}\overline{J}_{4,48}J_{12,48}J_{48}^2}{J_{1,2}\overline{J}_{16,48}\overline{J}_{8,48}J_{24}^3J_{96}} \nonumber \\ 
&- 2q^{-1}\frac{J_{8,32}J_{20,48}\overline{J}_{18,48}J_{2,24}\overline{J}_{20,48}J_{18,48}^2\overline{J}_{6,48}^2}{J_{1,2}\overline{J}_{16,48}\overline{J}_{8,48}J_{24}^3J_{48}J_{96}} + 2q^{10}\frac{J_{16,32}J_{4,48}\overline{J}_{6,48}J_{10,24}\overline{J}_{4,48}J_{12,48}J_{96}}{J_{1,2}\overline{J}_{16,48}\overline{J}_{8,48}J_{24}^2J_{48}} \nonumber  \\
&+ 2\frac{J_{16,32}J_{4,48}\overline{J}_{6,48}J_{10,24}\overline{J}_{20,48}J_{36,96}^2}{J_{1,2}\overline{J}_{16,48}\overline{J}_{8,48}J_{24}^2J_{96}}, \label{mock3-5} \\
\mathcal{M}_{18}(q) 
&:= \sum_{n \geq 0} \sum_{n \geq j \geq 0} \frac{(-1)^jq^{n^2+n+\binom{j}{2}}}{(q)_{n-j}(q)_{j}(1-q^{2j+1})} \nonumber \\
&= \frac{1}{(q)_{\infty}}f_{3,7,3}(q^2,q^5,q) \nonumber \\ 
&=  m(-q^{67}, q^{120}, q^{-9}) + q^{-9} m(-q^{13}, q^{120}, q^9) - q^{-2} m(-q^{37}, q^{120}, q^9) \nonumber \\
&- q^{-1} m(-q^{43}, q^{120}, q^{-9}) + q^{-7}\frac{J_{12,48}J_{3,40}J_{8,40}\overline{J}_{17,40}J_{6,20}\overline{J}_{12,40}J_{14,40}J_{40}^2}{J_1J_{9,120}\overline{J}_{18,40}\overline{J}_{6,40}J_{20}^3J_{80}} \nonumber \\
&+ q^{-4}\frac{J_{12,48}J_{3,40}J_{8,40}\overline{J}_{17,40}J_{6,20}\overline{J}_{8,40}J_{17,40}^2\overline{J}_{3,40}^2}{J_1J_{9,120}\overline{J}_{18,40}\overline{J}_{6,40}J_{20}^3J_{40}J_{80}} \nonumber \\
&- q^{-3}\frac{J_{24,48}J_{3,40}J_{12,40}\overline{J}_{3,40}J_{4,20}\overline{J}_{12,40}J_{14,40}J_{80}}{J_1J_{9,120}\overline{J}_{18,40}\overline{J}_{6,40}J_{20}^2J_{40}} \nonumber \\
&- q^{-7}\frac{J_{24,48}J_{3,40}J_{12,40}\overline{J}_{3,40}J_{4,20}\overline{J}_{8,40}J_{34,80}^2}{J_1J_{9,120}\overline{J}_{18,40}\overline{J}_{6,40}J_{20}^2J_{80}},  \label{mock3-6} \\
\mathcal{M}_{19}(q) 
&:= \sum_{n \geq 0} \sum_{n \geq j \geq 0} \frac{(-q)_n(-1)^jq^{\binom{n+1}{2}+\binom{j}{2}}}{(q)_{n-j}(q)_{j}(1-q^{2j+1})} \nonumber \\
&= \frac{(-q)_{\infty}}{(q)_{\infty}}f_{1,3,1}(q,q^4,q^2) \nonumber \\
&=  m(-q^{11}, q^{16}, q^{-3}) + q\frac{J_{4,8}J_{16,32}J_{5,16}J_{6,32}}{J_{1,2}\overline{J}_{8,16}\overline{J}_{2,16}}. \label{mock3-7}
\end{align}
\end{theorem}

It will have been noticed that some of the expressions in Theorems \ref{mockthm1}--\ref{mockthm3} are considerably more involved than others.   For instance, equation \eqref{mock3-6} involves four Appell-Lerch series and four modular forms while equation \eqref{mock3-7} involves only one of each. This depends on the indefinite theta function $f_{n,n+p,n}(x,y,q)$. In general, the number of Appell-Lerch series grows with $n$ and the number of modular forms grows with $p$.    

The final goal of the paper is to give identities involving some of the double sums in Theorems \ref{mockthm1}--\ref{mockthm3} and ``classical" mock theta functions.  Namely, we express the double sums $\mathcal{M}_{2}(q)$, $\mathcal{M}_{5}(q)$, $\mathcal{M}_{9}(q)$, and $\mathcal{M}_{16}(q)$ in terms of the mock theta functions

$$
T_{0}(q) := \sum_{n \geq 0} \frac{q^{(n+1)(n+2)} (-q^2;q^2)_{n}}{(-q; q^2)_{n+1}},
$$

$$
\omega(q) := \sum_{n \geq 0} \frac{q^{2n(n+1)}}{(q;q^2)_{n+1}^2},
$$

$$
A(q) := \sum_{n \geq 0} \frac{q^{n+1} (-q^2; q^2)_n}{(q;q^2)_{n+1}},
$$

\noindent and

$$
U_1(q) := \sum_{n \geq 0} \frac{q^{(n+1)^2} (-q;q^2)_n}{(-q^2;q^4)_{n+1}},
$$

\noindent of ``orders" 8, 3, 2, and 8, respectively (see \cite{GM1}). A similar identity was found in \cite{Lo-Os1}, namely

\begin{equation*} \label{id0}
\mathcal{W}_{2}(q)= 2q T_1(q) -q S_{1}(q)
\end{equation*}

\noindent where

$$
S_1(q) := \sum_{n \geq 0} \frac{q^{n(n+2)} (-q; q^2)_n}{(-q^2; q^2)_n}
$$

\noindent and

$$
T_1(q) := \sum_{n \geq 0} \frac{q^{n(n+1)} (-q^2; q^2)_n}{(-q; q^2)_{n+1}}
$$

\noindent are mock theta functions of order 8 \cite{GM1}.

\begin{corollary} \label{mockid} We have the following identities. 
\begin{align}
\mathcal{M}_2(q) &= 4T_0(q) + 2 - 2\frac{J_8^3J_{4,8}}{J_{2,8}^2\overline{J}_{1,8}} +  \frac{J_{2,4}J_{8,16}J_{1,8}J_{6,16}}{\overline{J}_{1,4}\overline{J}_{1,8}\overline{J}_{3,8}}, \\ \label{id1}
\mathcal{M}_5(q) &= 1 +q\omega(q), \\ \label{id2}
\mathcal{M}_{9}(q) &= -A(-q^{8}) + \frac{J_{32}^3J_{14,32}\overline{J}_{10,32}}{J_{16,32}J_{2,32}\overline{J}_{6,32}\overline{J}_{8,32}} - q^{-1}\frac{J_{64}^2J_{28,64}}{J_{32}J_{4,64}} + q^{-1}\frac{J_{8,16}J_{32,64}J_{4,32}J_{24,64}}{\overline{J}_{1,4}\overline{J}_{6,32}\overline{J}_{2,32}}, \\  \label{id3}
\mathcal{M}_{16}(q) &= \frac{1}{2} + q^{-1} U_1(q^8) - q^{-1} \frac{J_{32}^3 \overline{J}_{10,32} J_{14,32}}{\overline{J}_{16,32} J_{6,32} J_{8,32} \overline{J}_{2,32}}
+ \frac{J_{32}^3 J_{10,32} \overline{J}_{6,32}}{J_{6,32} J_{16,32} \overline{J}_{0,32} \overline{J}_{10,32}} \\
& + q^{-1} \frac{J_{8,16} J_{32,64} J_{4,32} J_{24,64}}{\overline{J}_{1,4} \overline{J}_{2,32} \overline{J}_{10,32}}.  \label{id4}
\end{align}
\end{corollary}

The paper is organized as follows.  In Section 2, we prove Theorems \ref{main1} and \ref{main2} and record some corollaries.  In Section 3, we establish a change of base lemma and deduce Andrews' Bailey pairs from those of Bringmann and Kane.  In Section 4, we recall important work of Hickerson and Mortenson on mock theta functions \cite{Hi-Mo1} and then prove Theorems \ref{mockthm1}--\ref{mockthm3} and Corollary \ref{mockid}.  

In \cite{realLo-Os}, we consider applications of Theorems \ref{main1}--\ref{main3} to $q$-hypergeometric double sums related to real quadratic fields, in the spirit of \cite{adh}.

\section{Proofs of Theorems \ref{main1} and \ref{main2}}
\begin{proof}[Proof of Theorem \ref{main1}]
First note that the sequence $\alpha_n'$ in (\ref{aeven}) and (\ref{aodd}) is uniquely defined by $\alpha_0' = 0$, $\alpha_1' = -(1-q^2)$, and
\begin{equation} \label{uniquedef}
\frac{\alpha_{n+2}'}{1-q^{2n+4}} - \frac{q^{2n}\alpha_n'}{1-q^{2n}} = -\alpha_{n+1}.
\end{equation}
Suppose that the $\beta_n'$ are given by (\ref{b1}). Then the corresponding $\alpha_n'$ satisfy the initial conditions.   Moreover, using \eqref{pairdefbis} we have
\begin{eqnarray*}
\frac{\alpha_{n+2}'}{1-q^{2n+4}} - \frac{q^{2n}\alpha_n'}{1-q^{2n}} &=& \sum_{j=1}^{n+2}\frac{(q)_{n+j+1}q^{\binom{n-j+2}{2}}(-1)^{n+j}}{(q)_{n-j+2}}\beta_j' - \sum_{j=1}^{n}\frac{(q)_{n+j-1}q^{\binom{n-j}{2}+2n}(-1)^{n+j}}{(q)_{n-j}}\beta_j' \\
&=& \sum_{j=1}^{n+2} \frac{(q)_{n+j-1}(-1)^{n+j}q^{\binom{n-j}{2}+2n}}{(q)_{n-j+2}}\Big((1-q^{n+j})(1-q^{n+j+1})q^{-2j+1}  \\ &\phantom{-}& \hskip2.5in -(1-q^{n-j+2})(1-q^{n-j+1})\Big)\beta_j' \\ 
&=& -\sum_{j=1}^{n+2}\frac{q^{\binom{n-j}{2}+2n}(-1)^{n+j}(q)_{n+j-1}}{(q)_{n-j+2}}\left((1-q^{2n+2})(1-q^{-2j+1})\right) \beta_j'\\
&=& (1-q^{2n+2})\sum_{j=0}^{n+1} \frac{q^{\binom{n-j-1}{2}+2n-2j-1}(q)_{n+j}(-1)^{n+j+1}}{(q)_{n-j+1}}(1-q^{2j+1})\beta_{j+1}' \\
&=& -(1-q^{2n+2})\sum_{j=0}^{n+1}\frac{q^{\binom{n+1-j}{2}}(q)_{n+j}(-1)^{n+j+1}}{(q)_{n+1-j}} \beta_j \\
&=& - \alpha_{n+1}.
\end{eqnarray*}
\end{proof}

\begin{proof}[Proof of Theorem \ref{main2}]
Let $a=q$ and let $\beta'_n$ be defined as in (\ref{b2}).  Then
\begin{eqnarray*}
\alpha'_n &=& \frac{(-1)^n}{1-q}\sum_{j=0}^n \frac{(q)_{n+j}(-1)^jq^{\binom{n-j}{2}}}{(q)_{n-j}}\beta'_j\left(1-q^{n+j+1} + q^{n+j+1}(1-q^{n-j})\right) \\
&=& \frac{(-1)^n}{1-q}\left(\sum_{j=0}^n\frac{(q)_{n+j+1}(-1)^jq^{\binom{n-j}{2}}}{(q)_{n-j}}\beta'_j + \sum_{j=0}^{n-1}\frac{(q)_{n+j}(-1)^jq^{\binom{n-j}{2}+n+j+1}}{(q)_{n-j-1}}\beta'_j\right) \\
&=& \frac{1}{1-q}\left(-\sum_{j=1}^{n+1}\frac{(q)_{n+j}(-1)^{n+j+1}q^{\binom{n-j+1}{2}}}{(q)_{n-j+1}}\beta_j + \sum_{j=1}^{n}\frac{(q)_{n+j-1}(-1)^{n-j}q^{\binom{n-j}{2}+2n}}{(q)_{n-j}}\beta_j\right) \\
&=& \frac{1}{1-q}\left(-\frac{\alpha_{n+1}}{1-q^{2n+2}} + \frac{q^{2n}\alpha_n}{1-q^{2n}}\right),
\end{eqnarray*}
which establishes the result.
\end{proof}
Note that the proof of Theorem \ref{main1} implies an inverse result.   Since the $\alpha_n'$ are uniquely defined by the $\alpha_n$ in \eqref{uniquedef} together with the initial conditions $\alpha_0' = 0$ and $\alpha_1' = -(1-q^2)$, we have the following.
\begin{theorem} \label{main1inverse}
If $(\alpha_n',\beta_n')$ form a Bailey pair relative to $1$ with $\alpha_0' = 0$ and $\alpha_1' = -(1-q^2)$, then $(\alpha_n,\beta_n)$
also form a Bailey pair relative to $1$, where $\alpha_0 = \beta_0 = 1$,
\begin{equation}
\alpha_n = \frac{-1}{1-q^{2n+2}}\alpha_{n+1}' + \frac{q^{2n-2}}{1-q^{2n-2}}\alpha_{n-1}',
\end{equation}
and
\begin{equation}
\beta_n = -(1-q^{2n+1})\beta_{n+1}'.
\end{equation}
\end{theorem}

We finish this section with three corollaries of Theorems \ref{main1} and \ref{main3}, giving three sets of two Bailey pairs involving indefinite quadratic forms.  These come from three Bailey pairs in Slater's list \cite{Sl1}.  These are not the only three pairs from Slater's list which lead to indefinite quadratic forms in Theorems \ref{main1} and \ref{main3}, but we have limited ourselves to those we will use in the sequel.

First, on p. 468 of \cite{Sl1} we find the Bailey pair relative to $1$,

\begin{equation*}
\alpha_n = 
\begin{cases}
1, &\text{if $n=0$}, \\
(-1)^n(q^n + q^{-n}), &\text{otherwise},
\end{cases}
\end{equation*}

\noindent and

\begin{equation*}
\beta_n = \frac{(-1)^nq^{-n}}{(q^2;q^2)_n}.
\end{equation*}

\noindent Applying Theorems \ref{main1} and \ref{main3} we have the following.  

\begin{corollary} \label{paircor1} 
The sequences $(a_n,b_n)$ form a Bailey pair relative to $1$, where

\begin{equation} \label{aevenslater1}
a_{2n} = (1-q^{4n})q^{2n^2-2n+1}\sum_{j=-n}^{n-1}q^{-2j^2},
\end{equation}

\begin{equation} \label{aoddslater1}
a_{2n+1} = -(1-q^{4n+2})q^{2n^2}\sum_{j=-n}^{n}q^{-2j^2-2j},
\end{equation}

\noindent and

\begin{equation} \label{bslater1}
b_n = 
\begin{cases}
0, &\text{if $n=0$},\\
\frac{(-1)^nq^{-n+1}}{(q^2;q^2)_{n-1}(1-q^{2n-1})}, &\text{otherwise},
\end{cases}
\end{equation}

\noindent and the sequences $(\alpha_n,\beta_n)$ form a Bailey pair relative to $q$, where

\begin{equation} \label{aevenslater1q}
\alpha_{2n} = \frac{1}{1-q}\left(q^{2n^2}\sum_{j=-n}^{n}q^{-2j^2-2j} + q^{2n^2+2n+1}\sum_{j=-n}^{n-1}q^{-2j^2}\right), 
\end{equation}

\begin{equation} \label{aoddslater1q}
\alpha_{2n+1} = -\frac{1}{1-q}\left(q^{2n^2+2n+1}\sum_{j=-n-1}^{n}q^{-2j^2} + q^{2n^2+4n+2}\sum_{j=-n}^{n}q^{-2j^2-2j}\right),
\end{equation}

\noindent and

\begin{equation} \label{bslater1q}
\beta_n = \frac{(-1)^nq^{-n}}{(q^2;q^2)_{n}(1-q^{2n+1})}.
\end{equation}

\end{corollary}

Next, on p. 468 of \cite{Sl1} we find the Bailey pair relative to $1$,
\begin{equation*}
\alpha_n = 
\begin{cases}
1, &\text{if $n=0$}, \\
(-1)^nq^{-\binom{n+1}{2}}(1+q^n), &\text{otherwise},
\end{cases}
\end{equation*}
and
\begin{equation*}
\beta_n = \frac{(-1)^nq^{-\binom{n+1}{2}}}{(q)_n}.
\end{equation*}

Applying Theorems \ref{main1} and \ref{main3} we have the following.   

\begin{corollary} \label{paircor2}
The sequences $(a_n,b_n)$ form a Bailey pair relative to $1$, where

\begin{equation} \label{aevenslater2}
a_{2n} = (1-q^{4n})q^{2n^2-2n}\sum_{j=-n}^{n-1}q^{-4j^2-3j},
\end{equation}

\begin{equation} \label{aoddslater2}
a_{2n+1} = -(1-q^{4n+2})q^{2n^2}\sum_{j=-n}^{n}q^{-4j^2-j},
\end{equation}

\noindent and

\begin{equation} \label{bslater2}
b_n = 
\begin{cases}
0, &\text{if $n=0$},\\
\frac{(-1)^nq^{-\binom{n}{2}}}{(q)_{n-1}(1-q^{2n-1})}, &\text{otherwise},
\end{cases}
\end{equation}

\noindent and the sequences $(\alpha_n,\beta_n)$ form a Bailey pair relative to $q$, where

\begin{equation} \label{aevenslater2q}
\alpha_{2n} = \frac{1}{1-q}\left(q^{2n^2}\sum_{j=-n}^{n}q^{-4j^2-j} + q^{2n^2+2n}\sum_{j=-n}^{n-1}q^{-4j^2-3j}\right), 
\end{equation}

\begin{equation} \label{aoddslater2q}
\alpha_{2n+1} = -\frac{1}{1-q}\left(q^{2n^2+2n}\sum_{j=-n-1}^{n}q^{-4j^2-3j} + q^{2n^2+4n+2}\sum_{j=-n}^{n}q^{-4j^2-j}\right),
\end{equation}

\noindent and

\begin{equation} \label{bslater2q}
\beta_n = \frac{(-1)^nq^{-\binom{n+1}{2}}}{(q)_{n}(1-q^{2n+1})}.
\end{equation}
\end{corollary}

Finally, on p. 468 of \cite{Sl1} we find the Bailey pair relative to $1$,
\begin{equation*}
\alpha_n = 
\begin{cases}
1, &\text{if $n=0$}, \\
(-1)^nq^{-n(n+3)/2}(1+q^{2n}), &\text{otherwise},
\end{cases}
\end{equation*}
and
\begin{equation*}
\beta_n = \frac{(-1)^nq^{-n(n+3)/2}}{(q)_n}.
\end{equation*}

Applying Theorems \ref{main1} and \ref{main3} we find the following.   

\begin{corollary} \label{paircor3}
The sequences $(a_n,b_n)$ form a Bailey pair relative to $1$, where

\begin{equation} \label{aevenslater3}
a_{2n} = (1-q^{4n})q^{2n^2-2n+1}\sum_{j=-n}^{n-1}q^{-4j^2-j}, 
\end{equation}

\begin{equation} \label{aoddslater3}
a_{2n+1} = -(1-q^{4n+2})q^{2n^2}\sum_{j=-n}^{n}q^{-4j^2-3j},
\end{equation}

\noindent and

\begin{equation} \label{bslater3}
b_n = 
\begin{cases}
0, &\text{if $n=0$},\\
\frac{(-1)^nq^{-\binom{n+1}{2}+1}}{(q)_{n-1}(1-q^{2n-1})}, &\text{otherwise},
\end{cases}
\end{equation}

\noindent and the sequences $(\alpha_n,\beta_n)$ form a Bailey pair relative to $q$, where

\begin{equation} \label{aevenslater3q} 
\alpha_{2n} = \frac{1}{1-q}\left(q^{2n^2}\sum_{j=-n}^{n}q^{-4j^2-3j} + q^{2n^2+2n+1}\sum_{j=-n}^{n-1}q^{-4j^2-j}\right), 
\end{equation}

\begin{equation} \label{aoddslater3q}
\alpha_{2n+1} = -\frac{1}{1-q}\left(q^{2n^2+2n+1}\sum_{j=-n-1}^{n}q^{-4j^2-j} + q^{2n^2+4n+2}\sum_{j=-n}^{n}q^{-4j^2-3j}\right),
\end{equation}

\noindent and

\begin{equation} \label{bslater3q}
\beta_n = \frac{(-1)^nq^{-n(n+3)/2}}{(q)_{n}(1-q^{2n+1})}.
\end{equation}
\end{corollary} 

\section{The seventh order Bailey pairs of Andrews} \label{Andrews'pairs}
We begin with a change of base lemma for Bailey pairs.  For other results of this nature, see \cite{Be-Wa1} and \cite{Br-Is-St1}.
Throughout this section we emphasize the base by saying that a pair of sequences satisfying \eqref{pairdef} is a Bailey pair relative to $(a,q)$.

\begin{lemma} \label{basechange}
If $(\alpha_n,\beta_n)$ is a Bailey pair relative to $(1,q)$, then $(\alpha_n',\beta_n')$ is a Bailey pair relative to $(1,q^2)$, where 
\begin{equation} \label{basechangealpha}
\alpha_n' = \frac{1}{2}(1+q^{2n})q^{n^2-n}\alpha_n
\end{equation} 
and
\begin{equation} \label{basechangebeta}
\beta_n' = \frac{1}{(-1)_{2n}}\sum_{k=0}^n\frac{q^{k^2-k}}{(q^2;q^2)_{n-k}}\beta_k.
\end{equation}
\end{lemma}

\begin{proof}
We will need the fact that
\begin{equation} \label{littlefact1}
(q^{-n})_k = \frac{(q)_n}{(q)_{n-k}}(-1)^kq^{\binom{k}{2} - nk}
\end{equation}
along with the identity
\begin{equation} \label{littlefact2}
\sum_{k=0}^{n-r}\frac{(-1)^kq^{2nk}(q^{-(n-r)})_k(-q^{-(n-r)})_k}{(q)_k(q^{2r+1})_k} = \frac{(1+q^{2r})(q)_{2r}(-1)_{2n}}{2(q^2;q^2)_{n+r}},
\end{equation}
which follows from a short calculation using the case $z= -q^{2n}$, $a=q^{-n+r}$, $b = -q^{-n+r}$, and $c = q^{2r+1}$ of the second Heine transformation \cite{Ga-Ra1},
\begin{equation}
\sum_{n \geq 0} \frac{(a)_n(b)_n}{(c)_n(q)_n}z^n = \frac{(c/b)_{\infty}(bz)_{\infty}}{(c)_{\infty}(z)_{\infty}}\sum_{n \geq 0} \frac{(\frac{abz}{c})_n(b)_n}{(bz)_n(q)_n} \left(\frac{c}{b}\right)^n.
\end{equation}

Now, beginning with \eqref{basechangebeta}, we have
\begin{align*}
\beta_n' &= \frac{1}{(-1)_{2n}}\sum_{k=0}^n\frac{q^{k^2-k}}{(q^2;q^2)_{n-k}}\beta_k \\
&= \frac{1}{(-1)_{2n}}\sum_{k=0}^n\frac{q^{k^2-k}}{(q^2;q^2)_{n-k}}\sum_{r=0}^k \frac{1}{(q)_{k-r}(q)_{k+r}}\alpha_r \\
&= \frac{1}{(-1)_{2n}}\sum_{r=0}^{n}\alpha_r \sum_{k=r}^{n} \frac{q^{k^2-k}}{(q^2;q^2)_{n-k}}\frac{1}{(q)_{k-r}(q)_{k+r}} \\
&= \frac{1}{(-1)_{2n}}\sum_{r=0}^{n}\alpha_r \sum_{k=0}^{n-r} \frac{q^{k^2+2kr+r^2-k-r}}{(q^2;q^2)_{n-k-r}(q)_{k}(q)_{k+2r}} \\
&= \frac{1}{(-1)_{2n}}\sum_{r=0}^{n}\frac{q^{r^2-r}}{(q)_{2r}}\alpha_r \sum_{k=0}^{n-r} \frac{q^{k^2+2kr-k}}{(q^2;q^2)_{n-k-r}(q)_{k}(q^{2r+1})_{k}} \\
&= \frac{1}{(-1)_{2n}}\sum_{r=0}^{n}\frac{q^{r^2-r}}{(q)_{2r}(q^2;q^2)_{n-r}}\alpha_r \sum_{k=0}^{n-r} \frac{(-1)^kq^{2nk}(q^{-2(n-r)};q^2)_k}{(q)_{k}(q^{2r+1})_{k}} \hskip.4in \text{(by \eqref{littlefact1})} \\
&= \frac{1}{(-1)_{2n}}\sum_{r=0}^{n}\frac{q^{r^2-r}}{(q)_{2r}(q^2;q^2)_{n-r}}\alpha_r \sum_{k=0}^{n-r} \frac{(-1)^kq^{2nk}(q^{-(n-r)})_k(-q^{-(n-r)})_k}{(q)_{k}(q^{2r+1})_{k}} \\
&= \sum_{r=0}^n \frac{(1+q^{2r}) q^{r^2-r}}{2(q^2;q^2)_{n-r}(q^2;q^2)_{n+r}}\alpha_r \hskip.4in \text{(by \eqref{littlefact2})}.
\end{align*}
This implies the statement of the theorem.
\end{proof}

Now we insert the Bailey pair in \eqref{a2n}--\eqref{bn} into Lemma \ref{basechange}.   We obtain a Bailey pair relative to $(1,q^2)$, where
\begin{equation*}
a_{2n}' = \frac{1}{2}(1-q^{8n})q^{6n^2-4n}\sum_{j=-n}^{n-1}q^{-2j^2-2j},
\end{equation*}
\begin{equation*}
a_{2n+1}' = -\frac{1}{2}(1-q^{8n+4})q^{6n^2+2n}\sum_{j=-n}^n q^{-2j^2},
\end{equation*}
$b_0' = 0$, and for $n \geq 1$,
\begin{align*}
b_n' &= \frac{1}{(-1)_{2n}}\sum_{k=1}^n\frac{(-1)^kq^{k^2-k}(q;q^2)_{k-1}}{(q)_{2k-1}(q^2;q^2)_{n-k}} \\
&= \frac{-1}{(-1)_{2n}}\sum_{k=0}^{n-1}\frac{(-1)^kq^{k^2+k}(q;q^2)_{k}}{(q)_{2k+1}(q^2;q^2)_{n-k-1}} \\
&= \frac{-1}{(-1)_{2n}(q^2;q^2)_{n-1}(1-q)}\sum_{k=0}^{n-1}\frac{q^{2nk}(q^{-2(n-1)};q^2)_k (q;q^2)_{k}}{(q^2;q^2)_k(q^3;q^2)_k} \hskip.4in \text{(by \eqref{littlefact1})} \\
&= \frac{-1}{(-1)_{2n}(q;q^2)_n} \\
&= \frac{-1}{2(q^{2n};q^2)_n},
\end{align*}
where the penultimate equality follows from the case $q=q^2$, $n= n-1$, $a=q$, and $c=q^3$ of the $q$-Chu-Vandermonde summation \cite{Ga-Ra1}
\begin{equation*}
\sum_{k= 0}^n \frac{(a)_k(q^{-n})_k}{(q)_k(c)_k}\left(\frac{cq^n}{a}\right)^k = \frac{(c/a)_n}{(c)_n}. 
\end{equation*}
Now multiplying $(a_n',b_n')$ by $-2$ and replacing $q$ by $q^{1/2}$ gives Andrews' Bailey pair \eqref{mathcalA2n(1)}--\eqref{mathcalBn(1)}. Finally, multiplying \eqref{mathcalA2n(1)}--\eqref{mathcalBn(1)} by $-1$, then applying Theorem \ref{main2} and Theorem \ref{main1inverse} gives the pairs \eqref{mathcalA2n(2)}--\eqref{mathcalBn(2)} and \eqref{mathcalA2n(0)}--\eqref{mathcalBn(0)}, respectively.

\section{Proofs of Theorems \ref{mockthm1}--\ref{mockthm3} and Corollary \ref{mockid}}

The approach for proving Theorems \ref{mockthm1}--\ref{mockthm3} is as follows.  We first apply (\ref{alphaprimedef}) and (\ref{betaprimedef}) to Corollaries \ref{paircor1}--\ref{paircor3} to obtain new Bailey pairs, then use (\ref{limitBailey}) in various ways to obtain identities expressing $q$-hypergeometric double sums in terms of the indefinite theta series (\ref{fdef}). Next, to deduce that these $q$-hypergeometric double sums are mock theta functions, we apply the following three explicit results of Hickerson and Mortenson which express (\ref{fdef}) in terms of the Appell-Lerch series (\ref{Appell-Lerch}). Define

\begin{equation} \label{g}
\begin{aligned}
g_{a,b,c}(x, y, q, z_1, z_0) & := \sum_{t=0}^{a-1} (-y)^t q^{c\binom{t}{2}} j(q^{bt} x, q^a) m\left(-q^{a \binom{b+1}{2} - c \binom{a+1}{2} - t(b^2 - ac)} \frac{(-y)^a}{(-x)^b}, q^{a(b^2 - ac)}, z_0 \right)\\
& + \sum_{t=0}^{c-1} (-x)^t q^{a \binom{t}{2}} j(q^{bt} y, q^c) m\left(-q^{c\binom{b+1}{2} - a\binom{c+1}{2} - t(b^2 -ac)}  \frac{(-x)^c}{(-y)^b}, q^{c(b^2 - ac)}, z_1\right).
\end{aligned}
\end{equation}

Following \cite{Hi-Mo1}, we use the term ``generic" to mean that the parameters do not cause poles in the Appell-Lerch sums or in the quotients of theta functions.

\begin{theorem}{\cite[Theorem 1.6]{Hi-Mo1}} \label{hm1} Let $n$ be a positive integer. For generic $x$, $y \in \mathbb{C}^{*}$

\begin{equation*}
f_{n, n+1, n}(x, y, q) = g_{n, n+1, n}(x, y, q, y^n / x^n, x^n / y^n).
\end{equation*}
\end{theorem}

\begin{theorem}{\cite[Theorem 1.9]{Hi-Mo1}} \label{hm2} Let $n$ be an odd positive integer. For generic $x$, $y \in \mathbb{C}^{*}$

\begin{equation*}
f_{n, n+2, n}(x, y, q) = g_{n, n+2, n}(x, y, q, y^n / x^n, x^n / y^n) - \Theta_{n,2}(x,y,q)
\end{equation*}

\noindent where

\begin{equation*}
\Theta_{n,2}(x,y,q) := \frac{y^{(n+1)/2} J_{2n, 4n} J_{4(n+1), 8(n+1)} j(y/x, q^{4(n+1)}) j(q^{n+2} xy, q^{4(n+1)}) j(q^{2n} / x^2 y^2, q^{8(n+1)})}{q^{(n^2-3)/2} x^{(n-3)/2} j(y^n / x^n, q^{4n(n+1)}) j(-q^{n+2} x^2, q^{4(n+1)}) j(-q^{n+2} y^2, q^{4(n+1)})}.
\end{equation*}

\end{theorem}

\begin{theorem}{\cite[Theorem 1.11]{Hi-Mo1}} \label{hm3} Let $n$ be an odd positive integer. For generic $x$, $y \in \mathbb{C}^{*}$

\begin{equation*}
f_{n, n+4, n}(x, y, q) = g_{n, n+4, n}(x, y, q, y^n / x^n, x^n / y^n) - \Theta_{n,4}(x,y,q)
\end{equation*}

\noindent where

\begin{equation*}
\Theta_{n,4}(x,y,q) := \frac{q^{-(n^2 + n -3)} x^{-(n-3)/2} y^{(n+1)/2} j(y/x, q^{4(2n+4)})}{j(y^n / x^n, q^{4n(2n+4)}) j(-q^{2n+8} x^4, q^{4(2n+4)}) j(-q^{2n+8}, q^{4(2n+4)})} \Bigl\{ J_{4n,16n} S_1 - qJ_{8n,16n} S_2 \Bigr \},
\end{equation*}

\begin{equation*}
\begin{aligned}
S_1 &:= \frac{j(q^{6n+16} x^2 y^2, q^{4(2n+4)}) j(-q^{2(2n+4)} y/x, q^{4(2n+4)}) j(q^{n+4} xy, q^{2(2n+4)})}{J_{2(2n+4}^3 J_{8(2n+4)}} \\
& \cdot \Biggl \{ j(-q^{2n+8} x^2 y^2, q^{4(2n+4)}) j(q^{2(2n+4)}, q^{4(2n+4)}) J_{4(2n+4)}^2 \\
& + \frac{q^{n+4} x^2 j(-q^{6n+16} x^2 y^2, q^{4(2n+4)}) j(q^{2(2n+4)} y/x, q^{4(2n+4)})^2 j(-y/x, q^{4(2n+4)})^2}{J_{4(2n+4}} \Biggr \}
\end{aligned}
\end{equation*}

\noindent and 

\begin{equation*}
\begin{aligned}
S_2 &:= \frac{j(q^{2n+8} x^2 y^2, q^{4(2n+4)}) j(-y/x, q^{4(2n+4)}) j(q^{3n+8} xy, q^{2(2n+4)})}{J_{2(2n+4)}^2} \\
& \cdot \Biggl \{ \frac{q^{n+1} j(-q^{2n+8} x^2 y^2, q^{4(2n+4)}) j(q^{2(2n+4)} y^2 / x^2, q^{4(2n+4)}) J_{8(2n+4)}}{y J_{4(2n+4)}} \\
& + \frac{qx j(-q^{6n+16} x^2 y^2, q^{4(2n+4)}) j(q^{4(2n+4)} y^2 / x^2, q^{8(2n+4)})^2}{J_{8(2n+4}} \Biggr \}.
\end{aligned}
\end{equation*}

\end{theorem}

Finally, we use the fact that specializations of Appell-Lerch series are well-known to be mock theta functions \cite{Za1}, \cite[Ch. 1]{Zw1}.

To simplify expressions arising in Theorems \ref{hm1}--\ref{hm3}, we require certain facts about $j(x,q)$ and $m(x,q,z)$. From the definition of $j(x,q)$, we have

\begin{equation} \label{j1}
j(q^{n} x, q) = (-1)^{n} q^{-\binom{n}{2}} x^{-n} j(x,q)
\end{equation}

\noindent where $n \in \mathbb{Z}$ and

\begin{equation} \label{j2}
j(x,q) = j(q/x, q) = -x j(x^{-1}, q).
\end{equation}

Next, some relevant properties of the sum $m(x, q, z)$ are given in the following (see (3.2b), (3.2c) of Proposition 3.1, (3.3) of Corollary 3.2 and Theorem 3.3 in \cite{Hi-Mo1}). 

\begin{proposition} \label{mprops} For generic $x$, $z$, $z_0$, $z{'} \in \mathbb{C}^{*}$,

\begin{equation} \label{m1}
m(x,q,z)=x^{-1} m(x^{-1}, q, z^{-1}),
\end{equation}

\begin{equation} \label{m2}
m(qx, q, z)=1-xm(x,q,z),
\end{equation}

\begin{equation} \label{mprod}
m(q, q^2, -1) = \frac{1}{2}
\end{equation}

\noindent and

\begin{equation} \label{m3}
m(x, q, z) = m(x, q, z_0) + \frac{z_0 J_1^3 j(z / z_{0}, q) j(x z z_{0}, q)}{j(z_0, q) j(z, q) j(xz_0, q) j(xz, q)}.
\end{equation}
\end{proposition}

We are now ready to proceed to the proofs of Theorems \ref{mockthm1}--\ref{mockthm3}.


\begin{proof}[Proof of Theorem \ref{mockthm1}]

Applying equations \eqref{alphaprimedef} and \eqref{betaprimedef} with $(a,\rho_1,\rho_2) = (1,-1,\infty)$ to (\ref{aevenslater1})--(\ref{bslater1}) and $(q,-q,\infty)$ to (\ref{aevenslater1q})--(\ref{bslater1q}) gives a Bailey pair relative to $1$, 

\begin{equation}  \label{a1n'even}
a'_{2n} = 2(1-q^{2n})q^{4n^2-n+1}\sum_{j=-n}^{n-1}q^{-2j^2}, 
\end{equation}

\begin{equation}  \label{a1n'odd}
a'_{2n+1} = -2(1-q^{2n+1})q^{4n^2+3n+1}\sum_{j=-n}^{n}q^{-2j^2-2j},
\end{equation}

\noindent and 

\begin{equation} \label{b1n'}
b'_n = \frac{1}{(-q)_n}\sum_{j=1}^n \frac{(-1)_j(-1)^jq^{\binom{j}{2}+1}}{(q)_{n-j}(q^2;q^2)_{j-1}(1-q^{2j-1})},
\end{equation}

\noindent and a Bailey pair relative to $q$,

\begin{equation}  \label{al1n'evenq}
\alpha'_{2n} = \frac{1}{1-q}\left(q^{4n^2+n}\sum_{j=-n}^{n}q^{-2j^2-2j} + q^{4n^2+3n+1}\sum_{j=-n}^{n-1}q^{-2j^2}\right), 
\end{equation}

\begin{equation} \label{a1n'oddq}
\alpha'_{2n+1} = -\frac{1}{1-q}\left(q^{4n^2+5n+2}\sum_{j=-n-1}^{n}q^{-2j^2} + q^{4n^2+7n+3}\sum_{j=-n}^{n}q^{-2j^2-2j}\right),
\end{equation}

\noindent and

\begin{equation} \label{be1n'}
\beta'_n = \frac{1}{(-q)_n}\sum_{j=0}^n \frac{(-1)^jq^{\binom{j}{2}}}{(q)_{n-j}(q)_{j}(1-q^{2j+1})},
\end{equation}

\noindent respectively. Now to prove \eqref{mock1-1}, we insert the Bailey pair $(a_n',b_n')$ from equations \eqref{a1n'even}--\eqref{b1n'} into \eqref{limitBailey} with $\rho_1$, $\rho_2 \to \infty$.   This gives 

\begin{equation*}
\begin{aligned}
q\mathcal{M}_{1}(q) & = \sum_{n \geq 0} q^{n^2} b_n'(q) = \frac{1}{(q)_{\infty}} \sum_{n \geq 0} q^{n^2} a_n'(q) \\
& = \frac{1}{(q)_{\infty}} \Biggl( \sum_{n \geq 0} q^{4n^2} a_{2n}'(q) + \sum_{n \geq 0} q^{4n^2 + 4n + 1} a_{2n+1}'(q) \Biggr) \\
& =  \frac{2}{(q)_{\infty}} \Biggl( \sum_{n \geq 0} q^{8n^2 - n+1} \sum_{j=-n}^{n-1} q^{-2j^2} - \sum_{n \geq 0} q^{8n^2 + n+1} \sum_{j=-n}^{n-1} q^{-2j^2} \\
& - \sum_{n \geq 0} q^{8n^2 + 7n + 2} \sum_{j=-n}^{n} q^{-2j^2 - 2j} + \sum_{n \geq 0} q^{8n^2 + 9n + 3} \sum_{j=-n}^{n} q^{-2j^2 - 2j} \Biggr). \\
\end{aligned}
\end{equation*}

\noindent After replacing $n$ with $-n$ in the second sum and $n$ with $-n-1$ in the fourth sum, we let $n=(r+s+1)/2$, $j=(r-s-1)/2$ in the first two sums and $n=(r+s)/2$, $j=(r-s)/2$ in the latter two sums to find

\begin{equation} \label{m1tof}
\begin{aligned}
\mathcal{M}_{1}(q)
& =  \frac{2q}{(q)_{\infty}} \Biggl( \Bigl(  \sum_{\substack{r, s \geq 0 \\ r \not\equiv s \imod{2} }} - \sum_{\substack{r, s < 0 \\ r \not\equiv {s \imod{2}} }} \Bigr)  q^{\frac{3}{2} r^2 + 5rs + \frac{9}{2} r + \frac{3}{2} s^2 + \frac{5}{2} s} \\
& - \Bigl(  \sum_{\substack{r, s \geq 0 \\ r \equiv s \imod{2} }} - \sum_{\substack{r, s < 0 \\ r \equiv {s \imod{2}} }} \Bigr) q^{\frac{3}{2} r^2 + 5rs + \frac{5}{2}r + \frac{3}{2}s^2 + \frac{9}{2}s} \Biggr) \\
& =  -\frac{2q}{(q)_{\infty}} \Biggl( \Bigl( \sum_{r,s \geq 0} - \sum_{r,s < 0} \Bigr) q^{\frac{3}{2} r^2 + 5rs + \frac{9}{2} r + \frac{3}{2} s^2 + \frac{5}{2} s} \Biggr)\\
& = -\frac{2q}{(q)_{\infty}} f_{3,5,3}(q^6, q^4, q).
\end{aligned}
\end{equation}

\noindent By (\ref{m1tof}), Theorem \ref{hm2} with $n=3$ and (\ref{g})--(\ref{j2}),

\begin{equation} \label{step1}
\begin{aligned}
\mathcal{M}_1(q)
& = -2q^{-10} m(-q^{-7}, q^{48}, q^6) + 2q^{-23} m(-q^{-23}, q^{48}, q^6) + 2m(-q^{25}, q^{48}, q^{-6}) \\
& - 2q^{-10} m(-q^{-7}, q^{48}, q^{-6}) + \frac{2q}{(q)_{\infty}} \Theta_{3,2}(q^6, q^4, q).
\end{aligned}
\end{equation} 

\noindent Now, we simplify (\ref{step1}) using (\ref{m1}), (\ref{m2}) and (\ref{m3}) to obtain

\begin{equation*}
\begin{aligned}
\mathcal{M}_1(q)
& = 4q^{-3} m(-q^7, q^{48}, q^{6}) + 4 m(-q^{25}, q^{48}, q^6) - 2 + \frac{2q}{(q)_{\infty}} \Theta_{3,2}(q^6, q^4, q) \\
& + 2q^{-3} \frac{J_{48}^3 J_{12,48} \overline{J}_{7,48}}{J_{6,48}^2 \overline{J}_{1,48} \overline{J}_{13,48}} + 2 \frac{J_{48}^3 J_{12,48} \overline{J}_{25,48}}{J_{6,48}^2 \overline{J}_{19,48} \overline{J}_{31,48}} \\
& = 4q^{-3} m(-q^7, q^{48}, q^{24}) + 4 m(-q^{25}, q^{48}, q^{-24}) - 4q^3\frac{J_{48}^3J_{18,48}\overline{J}_{11,48}}{J_{6,48}J_{24,48}\overline{J}_{13,48}\overline{J}_{17,48}} \\
& - 4\frac{J_{48}^3J_{18,48}\overline{J}_{7,48}}{J_{6,48}J_{24,48}\overline{J}_{1,48}\overline{J}_{17,48}} - 2 + \frac{2q}{(q)_{\infty}} \Theta_{3,2}(q^6, q^4, q) \\
& + 2q^{-3} \frac{J_{48}^3 J_{12,48} \overline{J}_{7,48}}{J_{6,48}^2 \overline{J}_{1,48} \overline{J}_{13,48}} + 2 \frac{J_{48}^3 J_{12,48} \overline{J}_{25,48}}{J_{6,48}^2 \overline{J}_{19,48} \overline{J}_{31,48}}.
\end{aligned}
\end{equation*}

\noindent Comparing with \eqref{mock1-1}, we are left with a modular identity to verify.  Such a verification can always be done using a finite computation.  This, and similar computations in this paper, were carried out using computer software packages available at

\begin{center}
\url{http://www.qseries.org/fgarvan}
\end{center}
This proves identity \eqref{mock1-1} and shows that $\mathcal{M}_{1}(q)$ is a mock theta function. 

As equations \eqref{mock1-2}--\eqref{mock1-5} are handled similarly, we briefly sketch the relevant details. For equations \eqref{mock1-2} and \eqref{mock1-3} we again use the Bailey pair \eqref{a1n'even}--\eqref{b1n'} in \eqref{limitBailey}, with $(\rho_1,\rho_2,q) = (\sqrt{q},-\sqrt{q},q)$ and $(q,\infty,q^2)$ and the above argument to obtain

\begin{equation*}
\mathcal{M}_{2}(q) =  \frac{(q;q^2)_{\infty}}{(q^2;q^2)_{\infty}}f_{1,3,1}(-q,-q^3,q)
\end{equation*}

\noindent and

\begin{equation*}
\mathcal{M}_{3}(q) =  \frac{2q(q;q^2)_{\infty}}{(q^2;q^2)_{\infty}}f_{1,2,1}(-q^5,-q^9,q^4).
\end{equation*}

\noindent One then proceeds with Theorems \ref{hm2} and \ref{hm1}, respectively, and simplifies. 




For \eqref{mock1-4} and \eqref{mock1-5} we use the Bailey pair \eqref{al1n'evenq}--\eqref{be1n'} in \eqref{limitBailey}, with $(\rho_1,\rho_2,q) = (\infty,\infty,q)$ and $(-q,\infty,q)$ to get

\begin{equation*}
\mathcal{M}_{4}(q) = \frac{1}{(q)_{\infty}}f_{3,5,3}(q^2,q^4,q)
\end{equation*}  	

\noindent and

\begin{equation*}
\mathcal{M}_{5}(q) = \frac{(-q)_{\infty}}{(q)_{\infty}} f_{1,2,1}(q,q^3,q^2).
\end{equation*}  	

\noindent Applying Theorems \ref{hm2} and \ref{hm1}, respectively, and continuing as above yields the result.
\end{proof}
	

\begin{proof}[Proof of Theorem \ref{mockthm2}]

Applying equations \eqref{alphaprimedef} and \eqref{betaprimedef} with $(a,\rho_1,\rho_2) = (1,\infty,\infty)$ to (\ref{aevenslater2})--(\ref{bslater2}) and $(q,\infty,\infty)$ to (\ref{aevenslater2q})--(\ref{bslater2q}) gives a Bailey pair relative to $1$, 

\begin{equation} \label{a2n'even}
a'_{2n} = (1-q^{4n})q^{6n^2-2n}\sum_{j=-n}^{n-1}q^{-4j^2-3j}, 
\end{equation}

\begin{equation} \label{a2n'odd}
a'_{2n+1} = -(1-q^{4n+2})q^{6n^2+4n+1}\sum_{j=-n}^{n}q^{-4j^2-j},
\end{equation}
and
\begin{equation} \label{b2n'}
b'_n = \sum_{j=1}^n \frac{(-1)^jq^{\binom{j+1}{2}}}{(q)_{n-j}(q)_{j-1}(1-q^{2j-1})},
\end{equation}
and a Bailey pair relative to $q$,

\begin{equation} \label{al2n'even}
\alpha'_{2n} = \frac{1}{1-q}\left(q^{6n^2+2n}\sum_{j=-n}^{n}q^{-4j^2-j} + q^{6n^2+4n}\sum_{j=-n}^{n-1}q^{-4j^2-3j}\right), 
\end{equation} 

\begin{equation} \label{al2n'odd}
\alpha'_{2n+1} = -\frac{1}{1-q}\left(q^{6n^2+8n+2}\sum_{j=-n-1}^{n}q^{-4j^2-3j} + q^{6n^2+10n+4}\sum_{j=-n}^{n}q^{-4j^2-j}\right),
\end{equation}
and
\begin{equation} \label{be2n'}
\beta'_n = \sum_{j=0}^n \frac{(-1)^jq^{\binom{j+1}{2}}}{(q)_{n-j}(q)_{j}(1-q^{2j+1})}, 
\end{equation}

\noindent respectively. For the identities \eqref{mock2-1}--\eqref{mock2-5} in Theorem \ref{mockthm2} we use the Bailey pair \eqref{a2n'even}--\eqref{b2n'} in \eqref{limitBailey} with $(\rho_1,\rho_2,q) = (\infty,\infty,q)$, $(-1,\infty,q)$, $(\sqrt{q},-\sqrt{q},q)$, $(q,\infty,q^2)$, and $(-1,-q,q^2)$ to obtain

\begin{equation*}
\mathcal{M}_6(q) =  -\frac{q^2}{(q)_{\infty}}f_{3,7,3}(q^5,q^6,q),
\end{equation*}

\begin{equation*}
\mathcal{M}_7(q) =  -\frac{2q^2(-q)_{\infty}}{(q)_{\infty}}f_{1,3,1}(q^4,q^5,q^2),
\end{equation*}
 
\begin{equation*}
 \mathcal{M}_8(q) = \frac{q(q;q^2)_{\infty}}{2(q^2;q^2)_{\infty}}f_{1,5,1}(-q^2,-q^3,q),
\end{equation*}
 
\begin{equation*}
 \mathcal{M}_9(q) = \frac{q^3(q;q^2)_{\infty}}{(q^2;q^2)_{\infty}}f_{1,3,1}(-q^7,-q^9,q^4),
 \end{equation*}
 
\begin{equation*}
 \mathcal{M}_{10}(q) = -\frac{2q^3(-q)_{\infty}}{(q)_{\infty}}f_{1,5,1}(q^5,q^7,q^2),
 \end{equation*}

\noindent respectively. One then applies Theorems \ref{hm2} and \ref{hm3} and proceeds as in the proof of (\ref{mock1-1}). The identities \eqref{mock2-6} and \eqref{mock2-7} follow similarly but with the Bailey pair \eqref{al2n'even}--\eqref{be2n'} in \eqref{limitBailey} with $(\rho_1,\rho_2,q) = (\infty,\infty,q)$ and $(-q,\infty,q)$ yielding

\begin{equation*}
\mathcal{M}_{11}(q) =  \frac{1}{(q)_{\infty}}f_{3,7,3}(q^3,q^4,q)
\end{equation*}

\noindent and

\begin{equation*}
\mathcal{M}_{12}(q) = \frac{(-q)_{\infty}}{(q)_{\infty}} f_{1,3,1}(q^2,q^3,q^2).
\end{equation*}

\noindent One now applies Theorems \ref{hm3} and \ref{hm2}, respectively, and simplifies.

\end{proof}


\begin{proof}[Proof of Theorem \ref{mockthm3}]

Applying \eqref{alphaprimedef} and \eqref{betaprimedef} with $(a,\rho_1,\rho_2) = (1,\infty,\infty)$ to (\ref{aevenslater3})--(\ref{bslater3})
and $(q,\infty,\infty)$ to (\ref{aevenslater3q})--(\ref{bslater3q}), we obtain a Bailey pair relative to $1$,

\begin{equation} \label{a3n'even}
a'_{2n} = (1-q^{4n})q^{6n^2-2n+1}\sum_{j=-n}^{n-1}q^{-4j^2-j}, 
\end{equation}

\begin{equation} \label{a3n'odd}
a'_{2n+1} = -(1-q^{4n+2})q^{6n^2+4n+1}\sum_{j=-n}^{n}q^{-4j^2-3j},
\end{equation}

\noindent and

\begin{equation} \label{b3n'}
b'_n = \sum_{j=1}^n \frac{(-1)^jq^{\binom{j}{2}+1}}{(q)_{n-j}(q)_{j-1}(1-q^{2j-1})},
\end{equation}

\noindent and a Bailey pair relative to $q$

\begin{equation}  \label{al3n'even}
\alpha'_{2n} = \frac{1}{1-q}\left(q^{6n^2+2n}\sum_{j=-n}^{n}q^{-4j^2-3j} + q^{6n^2+4n+1}\sum_{j=-n}^{n-1}q^{-4j^2-j}\right), 
\end{equation}

\begin{equation} \label{al3n'odd}
\alpha'_{2n+1} = -\frac{1}{1-q}\left(q^{6n^2+8n+3}\sum_{j=-n-1}^{n}q^{-4j^2-j} + q^{6n^2+10n+4}\sum_{j=-n}^{n}q^{-4j^2-3j}\right),
\end{equation}

\noindent and

\begin{equation} \label{be3n'}
\beta'_n = \sum_{j=0}^n \frac{(-1)^jq^{\binom{j}{2}}}{(q)_{n-j}(q)_{j}(1-q^{2j+1})}, 
\end{equation}

\noindent respectively. For the identities \eqref{mock3-1}--\eqref{mock3-5} in Theorem \ref{mockthm3} we use the Bailey pair \eqref{a3n'even}--\eqref{b3n'} in \eqref{limitBailey} with $(\rho_1,\rho_2,q) = (\infty,\infty,q)$, $(-1,\infty,q)$, $(\sqrt{q},-\sqrt{q},q)$, $(q,\infty,q^2)$, and $(-1,-q,q^2)$ to obtain

\begin{equation*}
\mathcal{M}_{13}(q) = -\frac{q}{(q)_{\infty}}f_{3,7,3}(q^4,q^7,q),
\end{equation*}

\begin{equation*}
\mathcal{M}_{14}(q) =  -\frac{2q(-q)_{\infty}}{(q)_{\infty}}f_{1,3,1}(q^3,q^6,q^2),
\end{equation*}

\begin{equation*}
\mathcal{M}_{15}(q) =  \frac{(q;q^2)_{\infty}}{2(q^2;q^2)_{\infty}}f_{1,5,1}(-q,-q^4,q),
\end{equation*}

\begin{equation*}
\mathcal{M}_{16}(q) =  \frac{q(q;q^2)_{\infty}}{(q^2;q^2)_{\infty}}f_{1,3,1}(-q^5,-q^{11},q^4),
\end{equation*}

\begin{equation*}
\mathcal{M}_{17}(q) = -\frac{2q(-q)_{\infty}}{(q)_{\infty}}f_{1,5,1}(q^3,q^9,q^2),
\end{equation*}

\noindent respectively. One then applies Theorems \ref{hm2} and \ref{hm3} and proceeds as above where (\ref{mprod}) is used for $\mathcal{M}_{16}(q)$. The identities \eqref{mock3-6} and \eqref{mock3-7} follow similarly but with the Bailey pair \eqref{al3n'even}--\eqref{be3n'} in \eqref{limitBailey} with $(\rho_1,\rho_2,q) = (\infty,\infty,q)$ and $(-q,\infty,q)$ yielding

\begin{equation*}
\mathcal{M}_{18}(q) = \frac{1}{(q)_{\infty}}f_{3,7,3}(q^2,q^5,q)
\end{equation*}

\noindent and

\begin{equation*}
\mathcal{M}_{19}(q) = \frac{(-q)_{\infty}}{(q)_{\infty}}f_{1,3,1}(q,q^4,q^2).
\end{equation*}

\noindent One applies Theorems \ref{hm3} and \ref{hm2}, respectively, and simplifies. 

\end{proof}


\begin{proof}[Proof of Corollary \ref{mockid}]

The identities in Corollary \ref{mockid} are established by comparing the expressions in Theorems \ref{mockthm1}--\ref{mockthm3} with those for classical mock theta functions.   We sketch the details. Equations (5.37) and (5.8) in \cite{Hi-Mo1} state 

$$
T_{0}(q)=-m(-q^3, q^8, q^2)
$$

\noindent and

$$
\omega(q) = -2q^{-1} m(q, q^6, q^2) + \frac{J_6^3}{J_2 J_{3,6}}.
$$

\noindent By (\ref{m1}), (\ref{m2}), (\ref{m3}), (\ref{mock1-2}) and (\ref{mock1-5}), (\ref{id1}) and (\ref{id2}) follow. Equations (5.1) and (5.40) in \cite{Hi-Mo1} state

$$
A(q)=-m(q,q^4,q^2)
$$

\noindent and 

$$
U_1(q)=-m(-q,q^4,-q^2).
$$

\noindent By (\ref{m3}), (\ref{mock2-4}) and (\ref{mock3-4}), (\ref{id3}) and (\ref{id4}) follow. 

\end{proof}

\section*{Acknowledgements} The authors thank Eric Mortenson for helping with the simplification of certain expressions in Theorems \ref{mockthm1}--\ref{mockthm3}. They also thank Dennis Stanton for information concerning Section 3.  The second author would like to thank the Institut des Hautes {\'E}tudes Scientifiques for their support during the completion of this paper.  This material is based upon work supported by the National Science Foundation under Grant No. 1002477.

\end{document}